\documentclass[reqno]{amsart}
\usepackage{amsmath,amssymb,amsfonts,caption}
\usepackage{graphicx}
\usepackage{pstricks}
\usepackage{euscript}
\usepackage{enumerate}
\usepackage{verbatim}
\usepackage{epsfig}

\newtheorem{theorem}{Theorem}

\newtheorem{definition}{Definition}

\newcommand{\eps}{\varepsilon}
\renewcommand{\phi}{\varphi}
\renewcommand{\le}{\leqslant}
\renewcommand{\ge}{\geqslant}
\newcommand{\ds}{\, ds}
\newcommand{\dr}{\, dr}

\newcommand{\dt}{\, dt}

\newcommand{\dint}{\displaystyle \int}

\DeclareMathOperator{\codim}{codim}
\DeclareMathOperator{\imagem}{Im}
\DeclareMathOperator{\sign}{sign}

\DeclareMathOperator{\e}{e}

\usepackage{dsfont}
   
   \newcommand{\R}{\ensuremath{\mathds R}}
\begin{document}
\title[Existence of periodic solutions for periodic eco-epidemic models]
   {Existence of periodic solutions for periodic eco-epidemic models with disease in the Prey}
\date{\today}
\author{C\'esar M. Silva}
\address{
   C\'esar M. Silva\\
   Departamento de Matem\'atica and CMA-UBI\\
   Universidade da Beira Interior\\
   6201-001 Covilh\~a\\
   Portugal}
\email{csilva@ubi.pt}
\thanks{C. Silva was partially supported by FCT through CMUBI (project UID/MAT/00212/2013)}
\begin{abstract}
For an eco-epidemic model with disease in the prey and periodic coefficients it is conjectured in [Xingge Niu, Tailei Zhang, Zhidong Teng, The asymptotic behavior of a nonautonomous eco-epidemic model with disease in the prey, Applied Mathematical Modelling 35, 457-470 (2011)] that, when the infected prey is permanent, there is a positive periodic orbit that is globally asymptotically stable in the interior of $(\R^+)^3$. In this paper we prove the existence part of the conjecture.
\end{abstract}
\maketitle
\section{Introduction}
 Lotka-Volterra models, that describe the predator-prey interaction, and epidemic models, that describe the spread of transmissible diseases among some population, are two major subjects of study in mathematical biology.

In the natural world, infected individuals become weaker and easier to be predated. Thus, in the predator-prey interaction, diseases cannot be ignored. Models that describe the spread of a disease in ecological systems are seldom referred to as eco-epidemiological models and are obtained by adding infected compartments to a Lotka-Volterra system. Several works concerning eco-epidemiological models have appeared recently: existence of Hopf bifurcations was studied in~\cite{Mukherjee-AMC-2010} and, for a model with stage structure, stability and existence of Hopf bifurcations were discussed in~\cite{Shi-Cui-Zhou-NA-2011}; for models with impulsive birth, it was established the existence of positive periodic solutions assuming that the birth pulse is strong enough in~\cite{Kang-Xue-Jin-JMAA-2008} and the existence of a globally stable prey eradication solution when the impulsive period is less than some critical value as well as a sufficient condition for permanence of the disease were obtained in~\cite{Liu-Zhang-Lu-JAMC-2012}; for a model with distributed time delay and impulsive control, conditions for the local and global asymptotical stability of the prey eradication periodic solution and the permanence of the disease were discussed in~\cite{Zou-Xiong-Shu-JFI-2011}. We emphasize that all the works above refer to models with constant parameters.

In this work we will consider the periodic version of the non-autonomous eco-epidemiological model already considered in~\cite{Niu-Zhang-Teng-AMM-2011}:
\begin{equation}\label{eq:ProblemaPrincipal}
\begin{cases}
S'=\Lambda(t)-\beta(t)\,SI-\mu(t) S\\
I'=\beta(t)\,SI-c(t)I-\eta(t)YI\\
Y'=Y(r(t)-b(t)Y+k(t)\eta(t)I)
\end{cases}.
\end{equation}
This model assumes that the total prey population is divided in two population classes: the class $S$ of susceptible prey and the class $I$ of infected prey. The remaining class, $Y$, corresponds to the predator population. The parameters in the model are the following: $\Lambda(t)$ is the recruited rate of the prey population (including newborns and migratory), $d(t)$ is the natural death rate of the prey population, $c(t)$ is the death rate among the infected prey population and includes the natural death rate and the disease-related death rate (naturally, $c(t) \ge d(t)$ for all $t$), $\beta(t)$ is the incidence rate, $r(t)$ is the intrinsic birth rate of the predators, $\eta(t)$ is the predation rate and $k(t)$ is the rate of converting prey into predator.

Several assumptions, reflected in the model's equations, were made. It was assumed that, in the absence of disease, the growth rate of the prey population is given by the solutions of $x'=\Lambda(t)-\mu(t)x$ and that the predator population grows according to a logistic curve. Another assumption is that the infected prey is removed by death or by predation before having the possibility of reproducing. Additionally, it is also assumed that the disease is not transmissible to the predator population, that the disease is not genetically inherited and that the infected population do not recover or become immune.

In~\cite{Niu-Zhang-Teng-AMM-2011}, conditions for the permanence and extinction of the infective prey as well as sufficient conditions for the global stability of the model were obtained. Also in that paper, in the end of section 4, the authors leave a conjecture. Namely, the authors conjecture that, if the parameter functions in~\eqref{eq:ProblemaPrincipal} are $\omega$-periodic, continuous and bounded functions on $\R_0^+$ (with all but $r$ being nonnegative and all but $r$, $\eta$ and $c$ having positive average) and if $\mathcal R>1$, where
 $$\mathcal R = \dfrac{\overline{\beta s_0}}{\overline c+\overline{\eta y_0}}$$
and $s_0$ and $y_0$ are the unique positive periodic solutions of $s'=\Lambda(t)-\mu(t)s$ and $y'=y(r(t)-b(t)y)$, then model~\eqref{eq:ProblemaPrincipal} has a positive periodic solution which is globally attractive in the interior of first octant. In what follows we will prove the existence part of this conjecture. Additionally, we will also consider the extinction situation.

It follows from the results in~\cite{Niu-Zhang-Teng-AMM-2011} that, in the periodic context, if $\mathcal R > 1$, the infected prey in system~\eqref{eq:ProblemaPrincipal} is permanent and, if $\mathcal R \le 1$, the infected prey in system~\eqref{eq:ProblemaPrincipal} goes to extinction. Thus $\mathcal R$ is a sharp threshold between permanence and extinction of the disease. In this work we prove that, in the periodic context, $\mathcal R$ is also a sharp threshold between existence of exactly two disease-free periodic orbits that contain the $\omega$-limit of all solutions and the coexistence of the referred disease free orbits (that become unstable) with at least one endemic periodic orbit.

The prove of our result on existence of periodic orbits relies on the famous Mawhin continuation theorem~\cite{Caines-Mawhin-NDE-1977}. To obtain our sharp result instead of a non-optimal condition, it is fundamental to use the permanence of the infectives already established in~\cite{Niu-Zhang-Teng-AMM-2011}. To discuss the stability of the disease-free orbit, we use the theory developed in~\cite{Wang-Zhao-JDDE-2008}.

\section{Existence and stability of disease-free periodic orbits}
Given a continuous and $\omega$-periodic function $f:\R_0^+ \to \R$, we define
$$\overline{f}=\frac{1}{\omega}\int_0^\omega f(s) \ds, \quad
f^u=\max_{t \in [0,\omega]} f(t) \quad \text{and} \quad
f^\ell=\min_{t \in [0,\omega]} f(t)$$

We assume that the following conditions hold:
\begin{enumerate}[C1)]
\item \label{eq:cond-1} $\Lambda$, $\beta$, $\mu$, $c$, $\eta$ and $k$ are $\omega$-periodic, nonnegative, continuous and bounded functions on $\R_0^+$ and $r$ is $\omega$-periodic, continuous and bounded function on $\R_0^+$;
\item \label{eq:cond-2} $\overline \Lambda>0$, $\overline \mu>0$, $\overline r>0$, $\overline b>0$ and $\overline \beta>0$.
\end{enumerate}

We also need to consider the auxiliary equations
\begin{equation}\label{sist-aux-1}
s'=\Lambda(t)-\mu(t)s
\end{equation}
and
\begin{equation}\label{sist-aux-2}
y'=y(r(t)-b(t)y).
\end{equation}

\begin{theorem}\label{teo:Df1}
  System~\eqref{eq:ProblemaPrincipal} has two disease-free periodic orbits, of period $\omega$, namely the orbits
  $\mathcal O_1$ and $\mathcal O_2$ given by
    $$(S_1(t),I_1(t),Y_1(t))=(s_0(t),0,0) \quad \text{and} \quad (S_2(t),I_2(t),Y_2(t))=(s_0(t),0,y_0(t)),$$
  where $s_0$ and $y_0$ are the unique positive periodic orbits of~\eqref{sist-aux-1} and~\eqref{sist-aux-2} respectively.
\end{theorem}

\begin{proof}
By Lemmas 1 and 3 in~\cite{Niu-Zhang-Teng-AMM-2011}, equations~\eqref{sist-aux-1} and~\eqref{sist-aux-2} have unique positive periodic solutions, respectively $s_0(t)$ and $y_0(t)$. These solutions have period $\omega$. Thus, it is easy to check that $(S_1(t),I_1(t),Y_1(t))=(s_0(t),0,0)$ and $(S_2(t),I_2(t),Y_2(t))=(s_0(t),0,y_0(t))$ are disease-free periodic solutions of system~\eqref{eq:ProblemaPrincipal} and have period $\omega$.
\end{proof}

Using the theory developed in~\cite{Wang-Zhao-JDDE-2008}, we will now determine a threshold for local stability/unstability of the disease-free orbit $\mathcal O_2$ of model~\eqref{eq:ProblemaPrincipal}.

It is easy to compute the matrices $F(t)$ and $V(t)$ in~\cite{Wang-Zhao-JDDE-2008} that in our context reduce to one dimensional matrices (that we identify with real numbers). In fact, for the orbit $\mathcal O_2$, we have $F(t)=\beta(t)s_0(t)$ and $V(t)=-c(t)-\eta(t)y_0(t)$.
The evolution operator $W(s,t,\lambda)$ associated with the linear $\omega$-periodic parametric system $w'=(-V(t)+F(t)/\lambda)w$ is easily seen to be given by
    $$W(s,t,\lambda)=\e^{-\int_s^t \beta(r)s_0(r)/\lambda -c(t)-\eta(r)y_0(r) \dr}$$
and thus
    $$W(\omega,0,\lambda)=1 \quad \Leftrightarrow \quad \overline{\beta s_0}/\lambda -\overline{c}-\overline{\eta y_0}=0 \quad \Leftrightarrow \quad
    \lambda = \frac{\overline{\beta s_0}}{\overline{c}+\overline{\eta y_0}}.$$
Define
    \begin{equation}\label{eq:mathcal-R}
        \mathcal R = \dfrac{\overline{\beta s_0}}{\overline c+\overline{\eta y_0}}.
    \end{equation}
Notice that $\mathcal R$ coincides with the threshold obtained in~\cite{Niu-Zhang-Teng-AMM-2011}. In fact, it follows from Corollaries 1 and 2 in~\cite{Niu-Zhang-Teng-AMM-2011} that when $\mathcal R>1$ the infected prey is permanent and, when $\mathcal R\le 1$ the infected prey goes to extinction.
It follows from Theorem 2.2 in~\cite{Wang-Zhao-JDDE-2008} and Corollary 2 in~\cite{Niu-Zhang-Teng-AMM-2011} that the periodic orbit $\mathcal O_2$ is locally asymptotically stable when $\mathcal R < 1$ and unstable when $\mathcal R > 1$. We can add the following result:

\begin{theorem}\label{teo:DF3}
  When $\mathcal R \le 1$, the $\omega$-limit of any solution of~\eqref{eq:ProblemaPrincipal} with initial condition $(t_0,S(t_0),I(t_0),Y(t_0))$ in the set $\{(t,S,I,Y) \in (\R_0^+)^4: Y>0\}$ is the periodic orbit $\mathcal O_2$ and the $\omega$-limit of any solution of~\eqref{eq:ProblemaPrincipal} with initial condition $(t_0,S(t_0),I(t_0),Y(t_0))$ in the set $\{(t,S,I,Y) \in (\R_0^+)^4: Y=0\}$ is the periodic orbit $\mathcal O_1$.
\end{theorem}

\begin{proof}
Let $\eps>0$ and $(S(t),I(t),Y(t))$ be some particular solution of~\eqref{eq:ProblemaPrincipal} with initial condition $(t_0,S(t_0),I(t_0),Y(t_0))$ and write $u(t)=S(t)-s_0(t)$ and $u_0=S(t_0)-s_0(t_0)$. Since $\mathcal R\le 1$, by Corollary 2 in~\cite{Niu-Zhang-Teng-AMM-2011}, there is $T>0$ such that $I(t)<\eps$ for all $t \ge T$.  Additionally, by Theorem 1 in~\cite{Niu-Zhang-Teng-AMM-2011}, there is $M_1>0$ such that $S(t)\le S(t)+I(t) \le M_1$. Thus
\[
u'=-\mu(t) u-\beta(t)SI \le -\mu(t) u-\beta(t) M_1 \eps,
\]
for all $t \ge T$. Therefore
\[
u(t) \le u_0 \e^{-\int_{t_0}^t \mu(s) \ds} +\beta^u M_1 \eps \int_{t_0}^t \e^{-\int_r^t \mu(s) \ds} \dr.
\]
By~C\ref{eq:cond-2}) and since $\eps>0$ is arbitrary, we conclude that $u(t)\to 0$ as $t \to +\infty$. Thus $S(t) \to s_0(t)$ as $t \to +\infty$.

Again by Theorem 1 in~\cite{Niu-Zhang-Teng-AMM-2011}, if $Y(t_0)>0$, there are $m_2,M_2>0$ such that, for $t$ sufficiently large, we have $m_2 \le Y(t)\le M_2$. Thus
\[
Y'=Y(r(t)-b(t)Y)+k(t)\eta(t)IY\le Y(r(t)-b(t)Y)+k(t)\eta(t)M_2 \eps,
\]
for all $t$ sufficiently large. Since $\eps>0$ is arbitrary, by Lemma 2 in~\cite{Niu-Zhang-Teng-AMM-2011}, we conclude that $|Y(t)-y_0(t)| \to 0$ as $t \to +\infty$.
We finally obtain that $(S(t),I(t),Y(t)) \to (s_0(t),0,y_0(t))$ as $t \to +\infty$.

Since the set $\mathcal C = \{(t,S,I,Y) \in (\R_0^+)^4: Y=0\}$ is invariant, it is immediate that the omega limit of any solution with initial condition in $\mathcal C$, $(S(t),I(t),0)$, verifies $(S(t),I(t),0) \to (s_0(t),0,0)$ as $t \to +\infty$.
\end{proof}

\begin{theorem}\label{teo:DF3}
  When $\mathcal R > 1$, the $\omega$-limit of any solution of~\eqref{eq:ProblemaPrincipal} with initial condition $(t_0,S(t_0),I(t_0),Y(t_0))$ in the set $\{(t,S,I,Y) \in (\R_0^+)^4: I=0 \wedge Y>0\}$ is the periodic orbit $\mathcal O_2$ and the $\omega$-limit of any solution of~\eqref{eq:ProblemaPrincipal} with initial condition $(t_0,S(t_0),I(t_0),Y(t_0))$ in the set $\{(t,S,I,Y) \in (\R_0^+)^4: I=Y=0\}$ is the periodic orbit $\mathcal O_1$.
\end{theorem}

\begin{proof}
According to Corollary 1 in~\cite{Niu-Zhang-Teng-AMM-2011}, when $\mathcal R>1$ the disease is permanent and thus the basin of attraction of the orbits $\mathcal O_1$ and $\mathcal O_2$ must be contained on the invariant set $\{(t,S,I,Y) \in (\R_0^+)^4: I=0\}$. On the other hand, if $(S(t),0,Y(t))$ is some solution of~\eqref{eq:ProblemaPrincipal}, it must satisfy $S'=\Lambda(t)-\mu(t)S$ and $Y'=Y(r(t)-b(t)Y)$. According to Lemmas 1 and 2 in~\cite{Niu-Zhang-Teng-AMM-2011}, the $\omega$-limit of nonnegative solutions of $S'=\Lambda(t)-\mu(t)S$ is the periodic orbit $\{s_0(t)\}$
and the $\omega$-limit of positive solutions of $Y'=Y(r(t)-b(t)Y)$ is the periodic orbit $\{y_0(t)\}$. Thus, the $\omega$-limit of orbits with initial conditions $(t_0,S(t_0),I(t_0),Y(t_0))$ in the set $\{(t,S,I,Y) \in (\R_0^+)^4: I=0 \wedge Y>0\}$ is the orbit $\mathcal O_2$. Moreover, it is immediate that the set $\{(t,S,I,Y) \in (\R_0^+)^4: I=Y=0\}$ is also invariant and that the $\omega$-limit of orbits in this set is the orbit $\mathcal O_1$. The result follows.
\end{proof}

\section{Existence of endemic periodic orbits}
When $\mathcal R>1$, it was proved in~\cite{Niu-Zhang-Teng-AMM-2011} (see Corollary 1) that we have permanence of the infected prey. In this section, we will use a well known result in degree theory, the Mawhin continuation theorem~\cite{Caines-Mawhin-NDE-1977}, and the permanence of the infected prey to establish the existence of at least one endemic periodic orbit for system~\eqref{eq:ProblemaPrincipal}.

\begin{theorem}\label{teo:main}
  When $\mathcal R>1$ system~\eqref{eq:ProblemaPrincipal} has an endemic $\omega$-periodic orbit.
\end{theorem}

As announced, to prove theorem~\ref{teo:main} we will use Mawhin's continuation theorem. Before stating this theorem, we need to give some definitions and recall some well known facts. Let $X$ and $Z$ be Banach spaces.
\begin{definition}
A linear mapping $\mathcal L: D \subseteq X \to Z$ is called a \emph{Fredholm mapping of index zero} if
\begin{enumerate}
\item $\dim \ker \mathcal L = \codim \imagem \mathcal L < \infty$;
\item $\imagem \mathcal L$ is closed in $Z$.
\end{enumerate}
\end{definition}
Given a Fredholm mapping of index zero, $\mathcal L: D \subseteq X \to Z$ , it is well known that there are continuous projectors $P:X\to X$ and $Q:Z\to Z$ such that
\begin{enumerate}
  \item $\imagem P = \ker \mathcal L$;
  \item $\ker Q = \imagem \mathcal L = \imagem (I-Q)$;
  \item $X = \ker \mathcal L \oplus \ker P$;
  \item $Z = \imagem \mathcal L \oplus \imagem Q$.
\end{enumerate}
It follows that $\mathcal L|_{D \cup \ker P}: (I-P)X \to \imagem \mathcal L$ is invertible. We denote the inverse of that map by $K_p$.
\begin{definition}
A continuous mapping $\mathcal N: X \to Z$ is called $L$-compact on $\overline{U} \subset X$, where $U$ is an open bounded set, if
\begin{enumerate}
\item $Q\mathcal N(\overline U)$ is bounded;
\item $K_p(I-Q)\mathcal N: \overline{U} \to X$ is compact.
\end{enumerate}
\end{definition}
Since $\imagem Q$ is isomorphic to $\ker \mathcal L$, there exists an isomorphism $\mathcal J: \imagem Q \to \ker \mathcal L$.

We are now prepared to state the theorem that will allow us to prove theorem~\ref{teo:main}.

\begin{theorem}(Mawhin's continuation theorem~\cite{Caines-Mawhin-NDE-1977}) \label{teo:Mawhin}
  Let $X$ and $Z$ be Banach spaces, let $U \subset X$ be an open set, let $\mathcal L: D \subseteq X \to Z$ be a  Fredholm mapping of index zero and let $\mathcal N: X \to Z$ be $L$-compact on $\overline{U}$. Assume that
\begin{enumerate}
  \item for each $\lambda \in (0,1)$ and $x \in \partial U \cap D$ we have $\mathcal L x \ne \lambda \mathcal N x$;
  \item for each  $x \in \partial U \cap \ker \mathcal L$ we have $Q\mathcal N x \ne 0$;
  \item $\deg(\mathcal J Q \mathcal N, U \cap \ker \mathcal L, 0) \ne 0$.
\end{enumerate}
Then the operator equation $\mathcal L x= \mathcal N x$ has at least one solution in $D \cap \overline U$.
\end{theorem}

We will now prove theorem~\ref{teo:main}.

\begin{proof}[Proof of theorem~\ref{teo:main}]
To apply Mawhin's theorem to our system we need to make a change of variables. Namely, with the change of variables $S(t)=\e^{u_1(t)}$, $I(t)=\e^{u_2(t)}$ and $Y(t)=\e^{u_3(t)}$, system~\eqref{eq:ProblemaPrincipal} becomes
\begin{equation}\label{eq:sistema-aplic-Mawhin}
\begin{cases}
u_1'=\Lambda(t)\e^{-u_1}-\beta(t)\e^{u_2}-\mu(t)\\
u_2'=\beta(t)\e^{u_1}-c(t)-\eta(t)\e^{u_3}\\
u_3'=r(t)-b(t)\e^{u_3}+k(t)\eta(t)\e^{u_2}
\end{cases}
\end{equation}
Note that, if $(u_1^*(t),u_2^*(t),u_3^*(t))$ is a periodic solution of period $\omega$ of system~\eqref{eq:sistema-aplic-Mawhin} then
$\left(\e^{u_1^*(t)},\e^{u_2^*(t)},\e^{u_3^*(t)}\right)$ is a periodic solution of period $\omega$ of system~\eqref{eq:ProblemaPrincipal}.

We will now prepare the setting where we will apply Mawhin's theorem. Our Banach spaces $X$ and $Z$ will be the space
$$X=Z=\{u = (u_1,u_2,u_3) \in C(\R,\R^3): u(t)=u(t+\omega) \}$$
with the norm
$$\|u\|=\max_{t \in [0,\omega]} |u_1(t)|+\max_{t \in [0,\omega]} |u_2(t)|+\max_{t \in [0,\omega]} |u_3(t)|.$$
Letting $D=X \cap C^1(\R,\R^3)$, we consider the linear map $\mathcal L: D \subseteq X \to Z$ given by
\[
\mathcal L u(t) = \dfrac{d u(t)}{dt}
\]
and the map $\mathcal N: X \to Z$ defined by
\[
\mathcal N u(t) =
\left[
\begin{array}{c}
\Lambda(t)\e^{-u_1}-\beta(t)\e^{u_2}-\mu(t)\\[2mm]
\beta(t)\e^{u_1}-c(t)-\eta(t))\e^{u_3}\\[2mm]
r(t)-b(t)\e^{u_3}+k(t)\eta(t)\e^{u_2}
\end{array}
\right].
\]
Consider also the projectors $P:X\to X$ and $Q:Z\to Z$ given by
$$ Pu = \dfrac{1}{\omega} \int_0^\omega u(t) \dt \quad \quad \text{and} \quad \quad
Qz = \dfrac{1}{\omega} \int_0^\omega z(t) \dt.$$
Note that $\imagem P = \ker \mathcal L = \R^3$, that
$$\ker Q = \imagem \mathcal L = \imagem (I-Q)
= \left\{z \in Z: \dfrac{1}{\omega} \int_0^\omega z(t) \dt =0\right\},$$
that $\mathcal L$ is a Fredholm mapping of index zero (since
$\dim \ker \mathcal L =\codim \imagem \mathcal L =3$) and that $\imagem \mathcal L$
is closed in $X$.

Consider the generalized inverse of $\mathcal L$,
$\mathcal K_p: \imagem \mathcal L \to D \cap \ker P$, given by
$$\mathcal K_p z(t)=\int_0^t z(s)\ds - \dfrac{1}{\omega} \int_0^\omega \int_0^r z(s) \ds \dr,$$
the operator $Q \mathcal N:X \to Z$ given by
\[
Q \mathcal N u(t) =
\left[
\begin{array}{c}
\dfrac{1}{\omega} \dint_0^\omega \Lambda(t)\e^{-u_1(t)}-\beta(t)\e^{u_2(t)}\dt -\bar\mu\\[3mm]
\dfrac{1}{\omega} \dint_0^\omega  \beta(t)\e^{u_1(t)}-\eta(t)\e^{u_3(t)} \dt-\bar c\\[3mm]
\dfrac{1}{\omega} \dint_0^\omega k(t)\eta(t)\e^{u_2(t)}-b(t)\e^{u_3(t)} \dt+\bar r
\end{array}
\right].
\]

and the mapping $\mathcal K_p (I-Q)\mathcal N:X \to D \cap \ker P$ given by
    $$\mathcal K_p (I-Q) \mathcal N u(t) = A_1(t)-A_2(t)-A_3(t)$$
where
\[
A_1(t)=
\left[
\begin{array}{c}
\dint_0^t \Lambda(t)\e^{-u_1(t)}-\beta(t)\e^{u_2(t)} - \mu(t) \dt \\[3mm]
\dint_0^t \beta(t)\e^{u_1(t)}-\eta(t)\e^{u_3(t)} - c(t) \dt \\[3mm]
\dint_0^t k(t)\eta(t)\e^{u_2(t)}-b(t)\e^{u_3(t)}+r(t) \dt
\end{array}
\right],
\]
\[
A_2(t)=
\left[
\begin{array}{c}
\dfrac{1}{\omega} \dint_0^\omega \dint_0^t \Lambda(s)\e^{-u_1(s)}-\beta(s)\e^{u_2(s)} - \mu(s) \ds \dt \\[3mm]
\dfrac{1}{\omega}\dint_0^\omega\dint_0^t\beta(s)\e^{u_1(s)}-\eta(s)\e^{u_3(s)}-c(s) \ds\dt\\[3mm]
\dfrac{1}{\omega} \dint_0^\omega \dint_0^t k(s)\eta(s)\e^{u_2(s)}-b(s)\e^{u_3(s)}+r(s) \ds \dt
\end{array}
\right]
\]
and
\[
A_3(t)=\left(\frac{t}{\omega}-\frac12\right)
\left[
\begin{array}{c}
\dint_0^\omega \Lambda(t)\e^{-u_1(t)}-\beta(t)\e^{u_2(t)} - \mu(t) \dt \\[3mm]
\dint_0^\omega \beta(t)\e^{u_1(t)}-\eta(t)\e^{u_3(t)} - c(t) \dt\\[3mm]
\dint_0^\omega k(t)\eta(t)\e^{u_2(t)}-b(t)\e^{u_3(t)}+r(t) \dt
\end{array}
\right].
\]

Let $\Omega\subset X$ be bounded. For any $u\in \Omega$, we have that $\|u\| \leq M$ and $|u_i(t)|\leq M$, where $i=1,2,3$. Next, we see that $Q\mathcal N (\overline\Omega)$ is bounded.

$$\left|\dfrac{1}{\omega} \dint_0^\omega \Lambda(t)\e^{-u_1(t)}-\beta(t)\e^{u_2(t)}\dt -\bar\mu\right| \leq \left|\bar\Lambda \e^{M} - \bar\beta \e^{-M} -\bar\mu\right| \leq \bar\Lambda e^{M} + \bar\beta e^{-M} +\bar\mu $$

$$\left|\dfrac{1}{\omega} \dint_0^\omega  \beta(t)\e^{u_1(t)}-\eta(t)\e^{u_3(t)} \dt-\bar c\right|\leq\bar\beta \e^{M} + \bar\eta\e^{-M}+\bar c$$

$$\left|\dfrac{1}{\omega} \dint_0^\omega k(t)\eta(t)\e^{u_2(t)}-b(t)\e^{u_3(t)} \dt+\bar r\right|\leq \overline{ k\eta}\e^{M} + \bar b\e^{-M} + \bar r$$

It is immediate that $Q\mathcal N$ and $\mathcal K_p (I-Q) \mathcal N$ are continuous.

Consider a sequence of function $\{u\}\subset \Omega$. We have the following inequality for the first function of $\mathcal K_p (I-Q) \mathcal N$.

$$
\begin{array}{rl}
[\mathcal K_p (I-Q) \mathcal N(u)]_{(1)}(t-v)=& \dint_v^t \Lambda(s)\e^{-u_1(s)}-\beta(s)\e^{u_2(s)} - \mu(s) \ds\\[3mm]
 &   - \left(\frac{t-v}{\omega}\right)\dint_0^\omega \Lambda(s)\e^{-u_1(s)}-\beta(s)\e^{u_2(s)} - \mu(s) \ds\\[3mm]
\leq & (t-v)(\Lambda^{u}\e^{M}-\beta^{l}\e^{-M}-\mu^{l})\\[3mm]
 & -(t-v)(\bar\Lambda\e^{-M}-\bar\beta\e^{M}-\bar\mu)
\end{array}
$$

For the other two functions, we have similar inequalities. Hence the sequence $\{\mathcal K_p (I-Q) \mathcal N(u)\}$ is equicontinuous. Using the periodicity of the functions, we know that the sequence $\{\mathcal K_p (I-Q) \mathcal N(u)\}$ is uniformly bounded.

An application of Ascoli-Arzela's theorem shows that $\mathcal K_p (I-Q) \mathcal N (\overline\Omega)$
is compact for any bounded set $\Omega \subset X$. Since $Q\mathcal N (\overline\Omega)$ is bounded, we conclude that $\mathcal N$
is $L$-compact on $\Omega$ for any bounded set $\Omega \subset X$.

Consider the system, for $\lambda\in (0,1)$
\begin{equation}\label{eq:sistema-aplic-Mawhin-lambda}
\begin{cases}
u_1'=\lambda\left(\Lambda(t)\e^{-u_1}-\beta(t)\e^{u_2}-\mu(t)\right) \\
u_2'=\lambda\left(\beta(t)\e^{u_1}-c(t)-\eta(t)\e^{u_3}\right) \\
u_3'=\lambda\left(r(t)-b(t)\e^{u_3}+k(t)\eta(t)\e^{u_2} \right)
\end{cases}
\end{equation}

Fix some $\lambda \in (0,1)$. Let $(u_1,u_2,u_3)$ be some solution of~\eqref{eq:sistema-aplic-Mawhin-lambda} and, for $i=1,2,3$, define
    $$u_i(\xi_i) = \min_{t \in [0,\omega]} u_i(t) \quad \quad \text{and} \quad \quad u_i(\chi_i) = \max_{t \in [0,\omega]} u_i(t)$$

It is clear that the derivatives of the function in its maximum and minimum have to be zero, i.e. $u_i'(\xi_i)=u_i'(\chi_i)=0$, for $i=1,2,3$.

By remark 1 in~\cite{Niu-Zhang-Teng-AMM-2011} and periodicity of $u$, for $t\geq 0$, we have
\begin{equation}\label{eq:maj_min_e^u1+e^u2}
A_1 \leq \min_{t \in [0,\omega]}\e^{u_2(t)}+\e^{u_1(t)}\\ \leq \e^{u_2(t)}+\e^{u_1(t)} \leq  \max_{t \in [0,\omega]}\e^{u_2(t)}+\e^{u_1(t)} \leq A_2,
\end{equation}

\begin{equation}\label{eq:maj_min_e^u3}
B_1  \le  \e^{u_3(\xi_3)}\leq \e^{u_3(t)} \le  \e^{u_3(\chi_3)}  \leq B_2,
\end{equation}
where $B_1=(r/b)^\ell$, $A_2=(\Lambda/\mu)^u$, $B_2=((r+k\eta A_2)/b)^u$ and $A_1=(\Lambda/[c+\eta B_2])^\ell$.

To bound $u_3(t)$ we simply note that $B_1\geq r^{^\ell}/b^u > 0$ and, by~\eqref{eq:maj_min_e^u3} we have
\begin{equation}\label{eq:maj-and-min_u3}
\theta_3^-:=\ln(B_1)\leq u_3(t)\leq \ln(B_2)=:\theta_3^+.
\end{equation}

By Theorem 2 in~\cite{Niu-Zhang-Teng-AMM-2011}, in our conditions we have permanence of the infected prey. Thus, by periodicity of $u$, we conclude that there is $m > 0$ such that $\e^{u_2(t)}\ge m$ for $t \ge 0$ (with $m$ independent of the chosen positive solution). Therefore, given $t\ge 0$ and any positive solution of~\eqref{eq:sistema-aplic-Mawhin-lambda} we have, for $t \ge 0$,
\begin{equation}\label{eq:maj-and-min_u2}
\theta_2^- := \ln m \leq u_2(t)\leq \ln A_2 =: \theta_2^+.
\end{equation}

By permanence of the infectives and the first equation in~\eqref{eq:ProblemaPrincipal}, we have
\[
S'=\Lambda(t)-\beta(t)SI-\mu(t)S\ge \Lambda(t) -(\beta(t) m +\mu(t))S \ge \Lambda^\ell -(\beta^u m +\mu^u)S
\]
and we conclude that
$$S(t) \ge \dfrac{\Lambda^\ell}{\beta^u m+\mu^u}+\left(S(0)+\dfrac{\Lambda^\ell}{\beta^u m+\mu^u}\right)\e^{-(\beta^u m +\mu^u)t}.$$
Thus, by periodicity, we must have
\[
\e^{u_1(t)} \ge \dfrac{\Lambda^\ell}{\beta^u m+\mu^u}
\]
and therefore
\begin{equation}\label{eq:maj-and-min_u1}
\theta_1^- := \ln ( \Lambda^\ell/(\beta^u m+\mu^u)) \leq u_1(t)\leq \ln A_2 =: \theta_1^+.
\end{equation}

Consider the algebraic equation
\begin{equation}
\begin{cases}
\bar\Lambda-\bar\beta\e^{p_1+p_2}-\bar\mu\e^{p_1} = 0\\
\bar\beta\e^{p_1}-\bar c-\bar\eta\e^{p_3} = 0\\
\bar r-\bar b\e^{p_3}+\overline{k\eta}\e^{p_2} = 0
\end{cases}.
\end{equation}
We claim that this equation has a unique solution $(p_1^*,p_2^*,p_3^*)$ in $(\R^+)^3$. In fact, by the second and third equations we get
\begin{equation}\label{eq:quadratic}
\dfrac{\overline{k\eta}\bar{\eta}}{\bar b}\e^{2p_2}+\left(\bar c+\dfrac{\bar\mu\bar\eta\overline{k\eta}}{\bar\beta\bar{b}}+\dfrac{\bar\eta \bar r}{\bar b}\right)\e^{p_2}+\dfrac{\bar\mu}{\bar \beta}\left(\bar c  + \dfrac{\bar \eta \bar r}{\bar b}\right)-\bar \Lambda=0.
\end{equation}
Since, by hypothesis, we have
\[
\dfrac{\bar\beta \bar\Lambda/\bar\mu}{\bar{c}+\bar\eta\bar{r}/\bar{b}}>1
\quad \Leftrightarrow \quad \bar \Lambda > \dfrac{\bar\mu}{\bar \beta}\left(\bar c  + \dfrac{\bar \eta \bar r}{\bar b}\right)
\]
we conclude that
\[
\sqrt{\left(\bar c+\dfrac{\bar\mu\bar\eta\overline{k\eta}}{\bar\beta\bar{b}}+
\dfrac{\bar\eta \bar r}{\bar b}\right)^2-4\dfrac{\overline{k\eta}\bar{\eta}}{\bar b}\left(\dfrac{\bar\mu}{\bar \beta}\left(\bar c  + \dfrac{\bar \eta \bar r}{\bar b}\right)-\bar \Lambda\right)} > \bar c+\dfrac{\bar\mu\bar\eta\overline{k\eta}}{\bar\beta\bar{b}}+
\dfrac{\bar\eta \bar r}{\bar b}
\]
and there is a unique positive solution of~\eqref{eq:quadratic}, $p_2^*$. Known $p_2^*$, we conclude that $p_3^*=\ln [(\bar r+\bar k\bar \eta \e^{p_2^*})/\bar b]$ and $p_1^*=\ln[(\bar c+\bar \eta \e^{p_3^*})/\bar \beta]$. The claim follows.

Let $M_0>0$ be such that $|p_1^*|+|p_2^*|+|p_3^*|<M_0$ and $M_i=\max\left\{|\theta^+_i|, |\theta^-_i|\right\}$, for $i=1,2,3$. And we define
\[
M=M_0+M_1+M_2+M_3
\]
We are now in conditions of establishing the region where we will apply Mahwin's theorem. We will consider
\[
\Omega=\{(u_1,u_2,u_3)\in X: \|(u_1,u_2,u_3)\|<M\}.
\]

The first two conditions of Theorem~\ref{teo:Mawhin} are satisfied by $\Omega$.

By condition~C\ref{eq:cond-2}) we have
$$
\begin{array}{rl}
 det( D ( Q \mathcal N)(p_1^*,p_2^*,p_3^*))= &\left|
\begin{array}{ccc}
-\bar\Lambda\e^{-p_1^*} & -\bar\beta\e^{p_2^*} & 0 \\[3mm]
 \bar\beta\e^{p_1^*}& 0 & -\bar\eta\e^{p_3^*} \\[3mm]
0 & \overline{k\eta}\e^{p_2^*} & -\bar b\e^{p_3^*}
\end{array}\right|\\[1cm]

=&  -(\overline\Lambda \, \overline k \, \overline{\eta}^2+\overline{b} \, \overline{\beta}^2 ) \e^{p_1^*+p_2^*+p_3^*}<0
\end{array}
$$
and therefore
$$
\deg(\mathcal J Q \mathcal N, U \cap \ker \mathcal L, 0) =\sign(\det( D ( Q \mathcal N)(p_1^*,p_2^*,p_3^*)))=-1 \ne 0.
$$
\end{proof}

\section{Simulation}

In this section we undertake some simulation to illustrate our results. We let $\Lambda(t)=0.7(1+0.9\cos(\pi+2\pi t))$, $\mu(t)=0.6(1+0.9\cos(2\pi t))$, $c(t)=0.1$, $b(t)=0.3(1+0.7\cos(\pi+2\pi t))$, $r(t)=0.2(1+0.7\cos(2\pi t))$, $k(t)=0.9$, $\eta(t)=0.7(1+0.7\cos(\pi+2\pi t))$ and $\beta(t)=\gamma(1+0.7\cos(2\pi t))$, where $\gamma>0$ is a parameter. We solve the differential equation and compute $\mathcal R$ numerically using MATHEMATICA and present some outputs below. First, we let $\gamma=0.45$ and considered the sets of initial conditions at time $t=0$:  $\xi_{df}=(S(0),I(0),Y(0))=(1.152,0,0.669)$, $\xi_1= (2,0.2,0.5)$ and $\xi_2=(0.1,0.6,0.7)$. Initial condition $\xi_{df}$ corresponds to the disease-free periodic orbit $\mathcal O_2$. For this situation we plotted $S(t)$ (left), $I(t)$ (center) and $Y(t)$ (right) in figure~\ref{fig1}. Numerically we obtained  $\mathcal R \approx 0.926\le 1$ witch, according to Corollary 2 in~\cite{Niu-Zhang-Teng-AMM-2011}, guarantees the extinction of the infectives and according to Theorem~\ref{teo:DF4}, that the solutions must approach the disease-free solution. We can see that the plots are consistent with these conclusions.
\begin{figure}[h!]
    \begin{minipage}[b][4cm]{.3\linewidth}
    \includegraphics[scale=0.38]{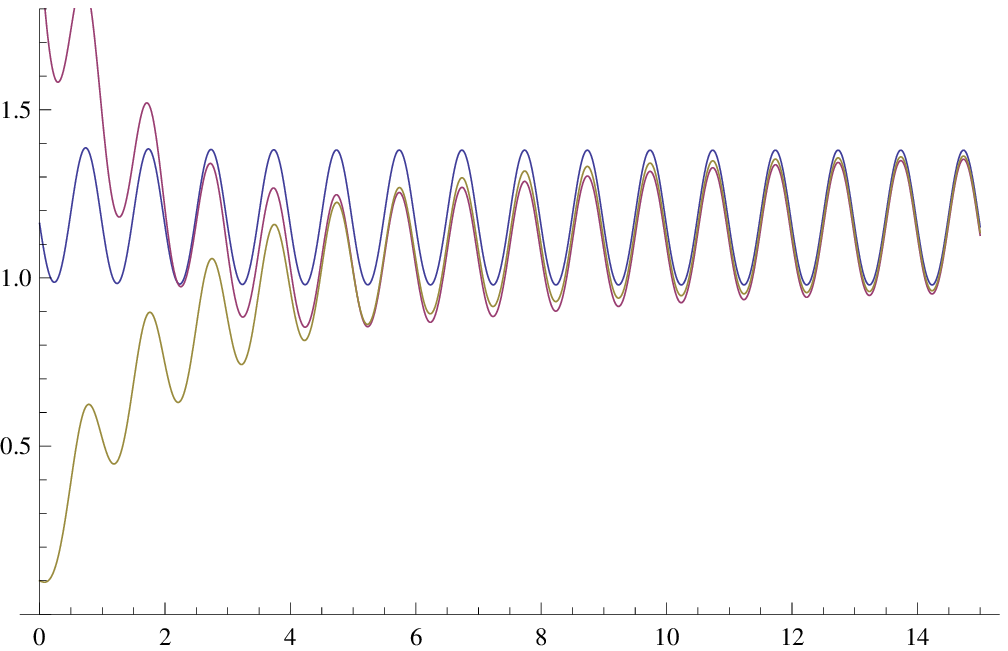}
  \end{minipage}
  \begin{minipage}[b][4cm]{.3\linewidth}
    \includegraphics[scale=0.38]{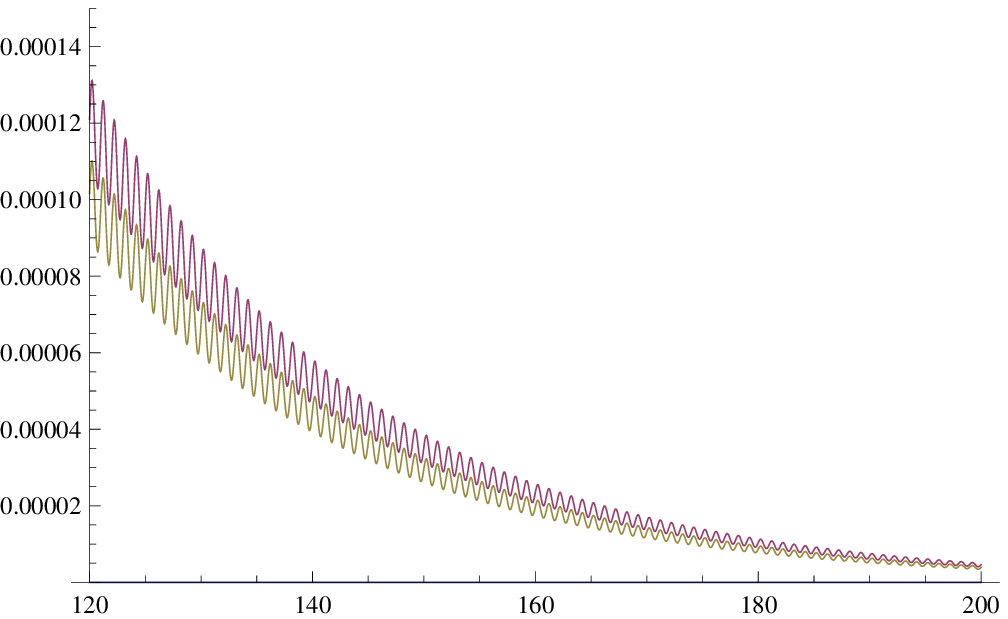}
  \end{minipage}
  \begin{minipage}[b][4cm]{.3\linewidth}
    \includegraphics[scale=0.38]{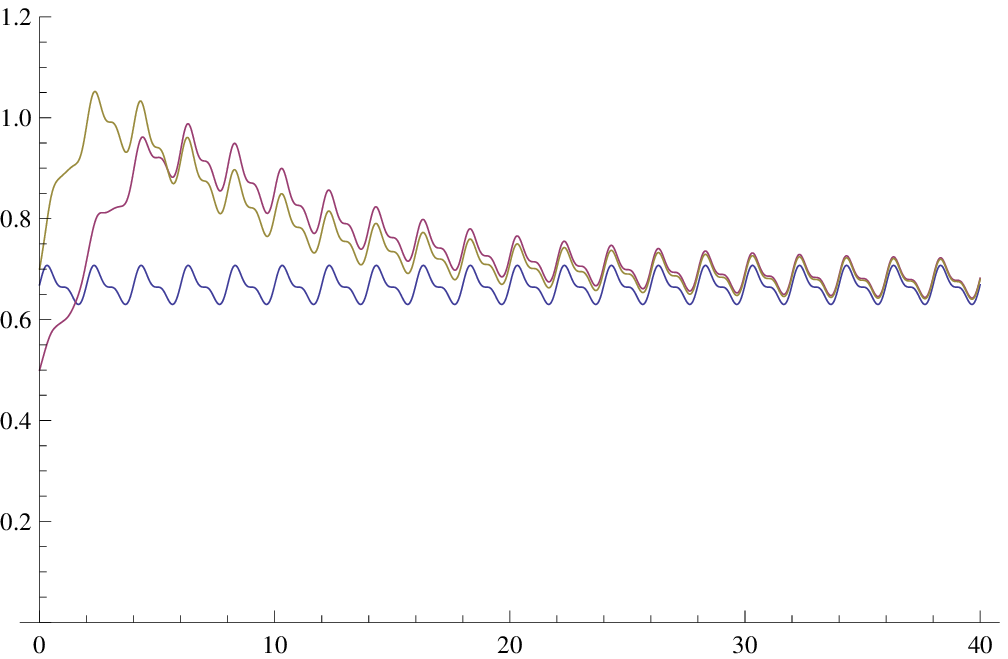}
  \end{minipage}
    \caption{Extinction}
      \label{fig1}
\end{figure}

Next, we let $\gamma=0.6$ and considered the sets of initial conditions at time $t=0$:  $\xi_e=(S(0),I(0),Y(0))=(1.082,0.065,0.799)$, $\xi_1= (2,0.2,0.5)$ and $\xi_2=(0.1,0.6,0.7)$. We plotted $S(t)$ (left), $I(t)$ (center) and $Y(t)$ (right) in figure~\ref{fig2}. Numerically we obtained  $\mathcal R \approx 1.238> 1$ witch, according to Corollary 1 in~\cite{Niu-Zhang-Teng-AMM-2011} implies the permanence of the infectives. We can see also that, for the initial condition $\xi_e$ we have an endemic periodic orbit, witch is consistent with the result in Theorem~\ref{teo:main}.
\begin{figure}[h!]
    \begin{minipage}[b][4cm]{.3\linewidth}
    \includegraphics[scale=0.38]{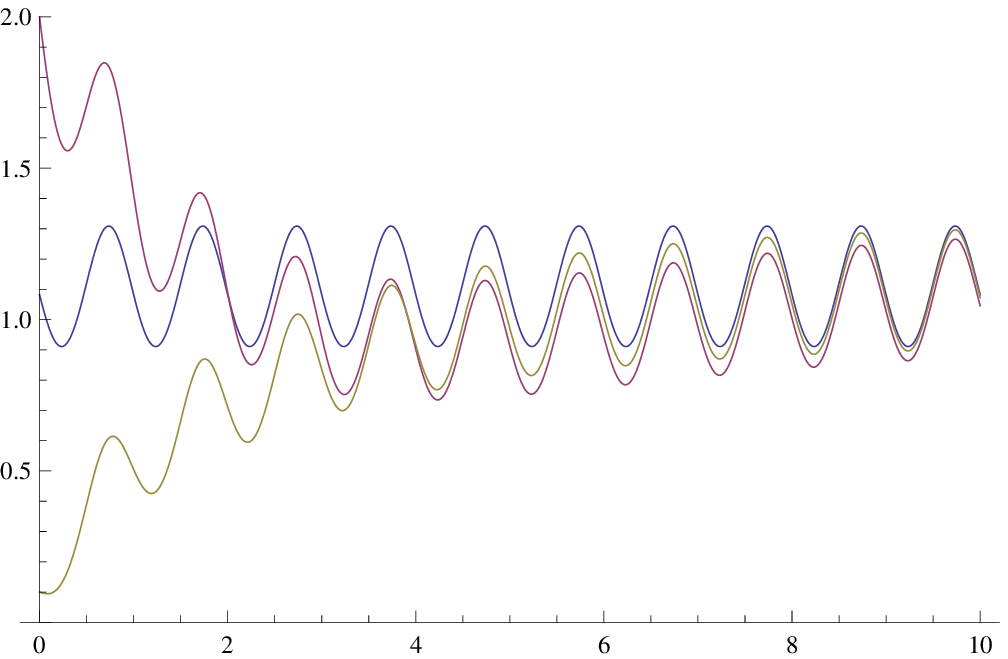}
  \end{minipage}
  \begin{minipage}[b][4cm]{.3\linewidth}
    \includegraphics[scale=0.38]{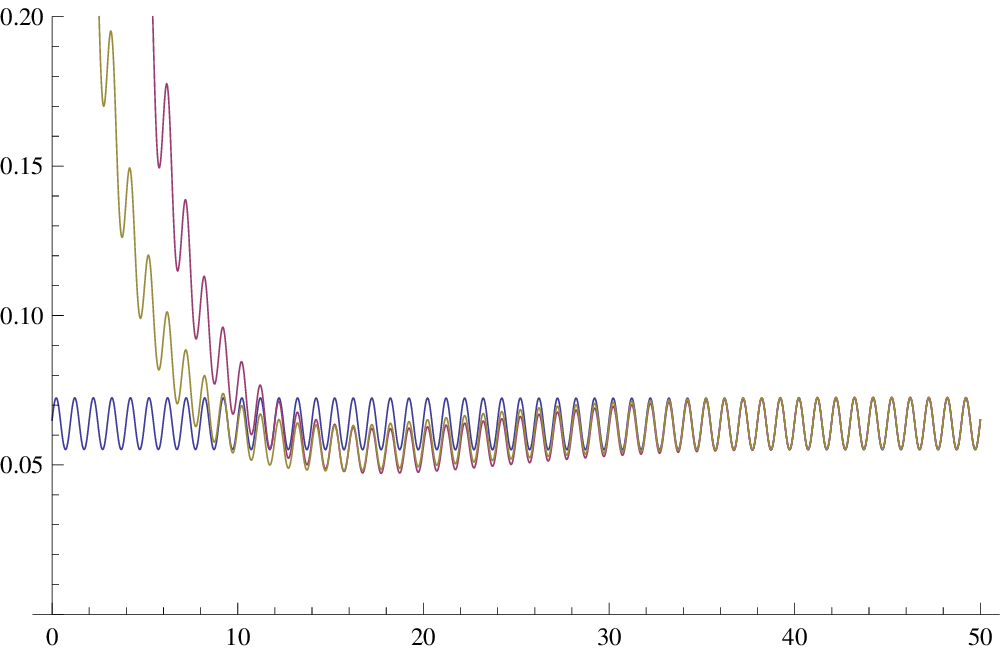}
  \end{minipage}
  \begin{minipage}[b][4cm]{.3\linewidth}
    \includegraphics[scale=0.38]{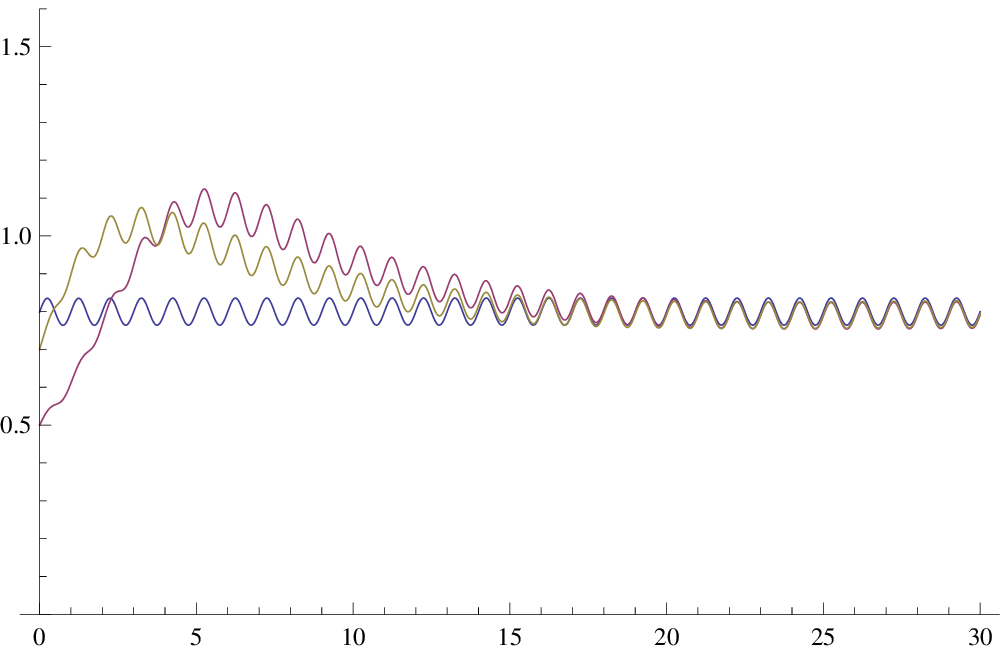}
  \end{minipage}
    \caption{Permanence}
      \label{fig2}
\end{figure}

\bibliographystyle{elsart-num-sort}

\end{document}